\documentclass[10pt]{llncs}  
\usepackage{amsmath,amssymb,amsfonts,graphicx, hyperref}

\date{February 8, 2010}

\newcommand{\defi}[1]{\textit{#1}}

\newcommand{\R}{\mathbb{R}}

\newcommand{\E}{\mathbb{E}}

\DeclareMathOperator{\ort}{O}

\DeclareMathOperator{\sdp}{SDP}
\DeclareMathOperator{\sign}{sign}
\DeclareMathOperator{\Tr}{Tr}

\begin{document}

\title{The positive semidefinite Grothendieck problem with rank constraint}

\author{
 Jop Bri\"et\inst{1}\thanks{The first author is supported by Vici grant 639.023.302 from the Netherlands Organization for Scientific Research (NWO), by the European Commission under the Integrated Project Qubit Applications (QAP) funded by the IST directorate as Contract Number 015848, and by the Dutch BSIK/BRICKS project.}, 
 Fernando M\'ario de Oliveira Filho\inst{2}\thanks{The second author was partially supported by CAPES/Brazil under grant BEX 2421/04-6, and the research was carried out at the Centrum Wiskunde \& Informatica, The Netherlands.} 
 and 
 Frank Vallentin\inst{3}\thanks{The third author is supported by Vidi grant 639.032.917 from the Netherlands Organization for Scientific Research (NWO).}
}

\institute{
 Centrum Wiskunde \& Informatica (CWI), Science Park 123, 1098 SJ Amsterdam, The Netherlands. \email{j.briet@cwi.nl}, 
\and
Department of Econometrics ad OR, Tilburg University, 5000 LE Tilburg, The Netherlands.
 \email{f.m.de.oliveira.filho@cwi.nl}
 \and
 Delft Institute of Applied Mathematics, Technical University of Delft, P.O. Box 5031, 2600 GA Delft, The Netherlands. \email{f.vallentin@tudelft.nl}
}

\maketitle

\begin{abstract}
Given a positive integer $n$ and a positive semidefinite matrix $A = (A_{ij}) \in \R^{m \times m}$, the positive semidefinite Grothendieck problem with rank-$n$-constraint $(\sdp_n)$ is
\begin{equation*} 
\text{maximize } \sum_{i=1}^m \sum_{j=1}^m A_{ij} \; x_i \cdot x_j,\qquad\text{where } x_1, \ldots, x_m \in S^{n-1}.
\end{equation*}
In this paper we design a randomized polynomial-time approximation algorithm for $\sdp_n$ achieving an approximation ratio of 
\begin{equation*}
\gamma(n) = \frac{2}{n}\left(\frac{\Gamma((n+1)/2)}{\Gamma(n/2)}\right)^2 = 1 - \Theta(1/n).
\end{equation*}
We show that under the assumption of the unique games conjecture the achieved approximation ratio is optimal: There is no polynomial-time algorithm which approximates $\sdp_n$ with a ratio greater than $\gamma(n)$.
We improve the approximation ratio of the best known polynomial-time algorithm for $\sdp_1$ from $2/\pi$ to $2/(\pi\gamma(m)) =  2/\pi + \Theta(1/m)$, and we show a tighter approximation ratio for $\sdp_n$ when $A$ is the Laplacian matrix of a graph with nonnegative edge~weights.
\end{abstract}

\markboth{J.~Bri\"et, F.M.~de Oliveira Filho, F.~Vallentin}{The positive semidefinite Grothendieck problem with rank constraint}

\section{Introduction}

Given a positive integer $n$ and a positive semidefinite matrix $A = (A_{ij}) \in \R^{m \times m}$,  the \defi{positive semidefinite Grothendieck problem with rank-$n$-constraint} is defined as
\begin{equation*}
\sdp_n(A) = \max\biggl\{\sum_{i=1}^m \sum_{j=1}^m A_{ij} \; x_i \cdot x_j : x_1, \ldots, x_m \in S^{n-1}\biggr\},
\end{equation*}
where $S^{n-1} = \{x \in \R^n : x \cdot x = 1\}$ is the unit
sphere. Note that the inner product matrix of the vectors $x_1,
\ldots, x_m$ has rank $n$.  This problem was introduced by Bri\"et,
Buhrman, and Toner \cite{BrietBuhrmanToner} in the context of quantum
nonlocality where they applied it to nonlocal XOR games.  The case $n
= 1$ is the classical positive semidefinite Grothendieck problem where
$x_1, \ldots, x_m \in \{-1,+1\}$. It was introduced by Grothendieck
\cite{Grothendieck} in the study of norms of tensor products of Banach
spaces.  It is an $\mathrm{NP}$-hard problem: If $A$ is the Laplacian
matrix of a graph then $\sdp_1(A)$ coincides with the value of a
maximum cut of the graph. The maximum cut problem (MAX CUT) is one of
Karp's 21 $\mathrm{NP}$-complete problems.  Over the last years, there
has been a lot of work on algorithmic applications, interpretations,
and generalizations of the Grothendieck problem and the companion
Grothendieck inequalities. For instance, Nesterov \cite{Nesterov}
showed that it has applications to finding and analyzing semidefinite
relaxations of nonconvex quadratic optimization problems. Ben-Tal and
Nemirovski \cite{BenTalNemirovski} showed that it has applications to
quadratic Lyapunov stability synthesis in system and control
theory. Alon and Naor \cite{AlonNaor} showed that it has applications
to constructing Szemer\'edi partitions of graphs and to estimating the
cut norm of matrices. Linial and Shraibman \cite{LinialSchraibman}
showed that it has applications to finding lower bounds in
communication complexity. Khot and Naor \cite{KhotNaor},
\cite{KhotNaor2} showed that it has applications to kernel
clustering. For other applications, see also Alon, Makarychev,
Makarychev, and Naor \cite{AlonMakarychevMakarychevNaor}, and
Raghavendra and Steurer \cite{RaghavendraSteurer}.

One can reformulate the positive semidefinite Grothendieck problem with rank-$n$-constraint as a semidefinite program with an additional rank constraint:
\begin{equation*}
\begin{split}
\text{maximize } & \sum_{i = 1}^m \sum_{j = 1}^m A_{ij} X_{ij}\\
\text{subject to } & \text{$X = (X_{ij}) \in \R^{m \times m}$ is positive semidefinite,}\\
& X_{ii} = 1,\quad\text{for $i = 1, \ldots, m$,}\\
& \text{$X$ has rank at most $n$.}
\end{split}
\end{equation*}
When $n$ is a constant that does not depend on the matrix size $m$ there is no polynomial-time algorithm known which solves $\sdp_n$.  It is also not known if the problem $\sdp_n$ is $\mathrm{NP}$-hard when $n \geq 2$.  On the other hand the \defi{semidefinite relaxation} of $\sdp_n(A)$ defined by
\begin{equation*}
\sdp_{\infty}(A) = \max\biggl\{\sum_{i = 1}^m \sum_{j = 1}^m A_{ij} \; u_i \cdot u_j : u_1, \ldots, u_m \in S^{\infty}\biggr\}
\end{equation*}
can be computed in polynomial time to any desired precision by using,
e.g., the ellipsoid method. Here $S^{\infty}$ denotes the unit sphere
of the Hilbert space $l^2(\R)$ of square summable sequences, which
contains $\R^n$ as the subspace of the first $n$ components. Clearly,
it would suffice to use unit vectors in $\R^m$ for solving
$\sdp_{\infty}(A)$ when $A \in \R^{m \times m}$, but using
$S^{\infty}$ will simplify many formulations in this paper.  Rietz
\cite{Rietz} (in the context of the Grothendieck inequality) and
Nesterov \cite{Nesterov} (in the context of approximation algorithms
for $\mathrm{NP}$-hard problems) showed that $\sdp_1$ and
$\sdp_{\infty}$ are always within a factor of at most $2/\pi$ from
each other. That is, for all positive semidefinite matrices $A \in
\R^{m \times m}$ we have
\begin{equation}
\label{eq:rietznesterov}
1 \geq \frac{\sdp_1(A)}{\sdp_{\infty}(A)} \geq \frac{2}{\pi}.
\end{equation}

By exhibiting an explicit series of positive semidefinite matrices,
Grothen\-dieck~\cite{Grothendieck} (see also Alon and Naor
\cite[Section 5.2]{AlonNaor}) showed that one cannot improve the
constant $2/\pi$ to $2/\pi + \varepsilon$ for any positive
$\varepsilon$ which is independent of~$m$. Nesterov~\cite{Nesterov}
gave a randomized polynomial-time approximation algorithm for $\sdp_1$
with approximation ratio $2/\pi$ which can be derandomized using the
techniques presented by Mahajan and Ramesh~\cite{MahajanRamesh}. This
algorithm is optimal in the following sense: Khot and Naor
\cite{KhotNaor} showed that under the assumption of the unique games
conjecture (UGC) there is no polynomial-time algorithm which
approximates $\sdp_1$ to within a ratio of~$2/\pi + \varepsilon$ for any
positive $\varepsilon$ independent of~$m$. The unique games conjecture
was introduced by Khot \cite{Khot} and by now many tight UGC hardness
results are known, see e.g.\ Khot, Kindler, Mossel, and O'Donnell
\cite{KhotKindlerMosselODonnel} for the maximum cut problem, Khot and
Regev \cite{KhotRegev} for the minimum vertex cover problem, and
Raghavendra \cite{Raghavendra} for general constrained satisfaction
problems.  The aim of this paper is to provide a corresponding
analysis for $\sdp_n$.

\subsection*{Our results}

In Section~\ref{sec:psd} we start by reviewing our methodological contributions:  Our main contribution is the analysis of a rounding scheme which can deal with rank-$n$-constraints in semidefinite programs. For this we use the Wishart distribution from multivariate statistics (see e.g.\ Muirhead \cite{Muirhead}). We believe this analysis is of independent interest and will turn out to be useful in different contexts, e.g.\ for approximating low dimensional geometric embeddings.  Our second contribution is that we improve the constant in inequality \eqref{eq:rietznesterov} slightly by considering functions of positive type for the unit sphere $S^{m-1}$ and applying a characterization of Schoenberg \cite{Schoenberg}. This slight improvement is the key for our UGC hardness result of approximating $\sdp_n$ given in Theorem~\ref{th:hardness}.  We analyze our rounding scheme in Section~\ref{sec:algorithm}.

\begin{theorem}
\label{th:algorithm}
For all positive semidefinite matrices $A \in \R^{m \times m}$ we have
\begin{equation*}
1 \geq \frac{\sdp_n(A)}{\sdp_{\infty}(A)} \geq \gamma(n) = \frac{2}{n}\left(\frac{\Gamma((n+1)/2)}{\Gamma(n/2)}\right)^2 = 1 - \Theta(1/n),
\end{equation*}
and there is a randomized polynomial-time approximation algorithm for $\sdp_n$ achieving this ratio.
\end{theorem}
The first three values of $\gamma(n)$ are: 
\begin{equation*}
\begin{split}
& \gamma(1) = 2/\pi = 0.63661\ldots\\
& \gamma(2) =\pi/4 = 0.78539\ldots\\
& \gamma(3) = 8/(3\pi)= 0.84882\ldots
\end{split}
\end{equation*}

In Section~\ref{sec:improved} we show that one can improve inequality~\eqref{eq:rietznesterov} slightly:

\begin{theorem}
\label{th:improvement}
For all positive semidefinite matrices $A \in \R^{m \times m}$ we have
\begin{equation*} 
1 \geq \frac{\sdp_1(A)}{\sdp_{\infty}(A)} \geq \frac{2}{\pi\gamma(m)} = \frac{m}{\pi} \left(\frac{\Gamma(m/2)}{\Gamma((m+1)/2)}\right)^2 = \frac{2}{\pi} + \Theta\left(\frac{1}{m}\right),
\end{equation*}
and there is a polynomial-time approximation algorithm for $\sdp_1$ achieving this ratio.
\end{theorem}

With this, the current complexity status of the problem $\sdp_1$ is similar to the one of the minimum vertex cover problem. Karakostas \cite{Karakostas} showed that one can approximate the minimum vertex cover problem for a graph having vertex set $V$ with an approximation ratio of $2-\Theta(1/\sqrt{\log |V|})$ in polynomial time. On the other hand, Khot and Regev \cite{KhotRegev} showed, assuming the unique games conjecture, that there is no polynomial-time algorithm which approximates the minimum vertex cover problem with an approximation factor of $2 - \varepsilon$ for any positive $\varepsilon$ which is independent of~$|V|$.
In Section~\ref{sec:hardness} we show that the approximation ratio $\gamma(n)$ given in Theorem~\ref{th:algorithm} is optimal for $\sdp_n$ under the assumption of the unique games conjecture.  By using the arguments of the proof of Theorem~\ref{th:improvement} and by the UGC hardness of approximating $\sdp_1$ due to Khot and Naor \cite{KhotNaor} we get the following tight UGC hardness result for approximating $\sdp_n$.

\begin{theorem}
\label{th:hardness}
Under the assumption of the unique games conjecture there is no polynomial-time algorithm which approximates $\sdp_n$ with an approximation ratio greater than~$\gamma(n) + \varepsilon$ for any positive $\varepsilon$ which is independent of the matrix size $m$.
\end{theorem}

In Section~\ref{sec:Laplacian} we show that a better approximation ratio can be achieved when the matrix $A$ is the Laplacian matrix of a graph with nonnegative edge weights.

\section{Rounding schemes and functions of positive type}
\label{sec:psd}

In this section we discuss our rounding scheme which rounds an optimal solution of $\sdp_{\infty}$ to a feasible solution of $\sdp_n$. In the case $n = 1$ our rounding scheme is equivalent to the classical scheme of Goemans and Williamson \cite{GoemansWilliamson}. To analyze the rounding scheme we use functions of positive type for unit spheres. 
The randomized polynomial-time approximation algorithm which we use in the proofs of the theorems is the following three-step process. The last two steps are our rounding scheme.
\begin{enumerate}
\item Solve $\sdp_{\infty}(A)$, obtaining vectors $u_1, \ldots, u_m \in S^{m-1}$.
\item Choose $X = (X_{ij}) \in \R^{n \times m}$ so that every matrix entry $X_{ij}$ is distributed independently according to the standard normal distribution with mean $0$ and variance $1$, that is,~$X_{ij} \sim N(0,1)$.
\item Set $x_i = Xu_i/\|Xu_i\|\in S^{n-1}$ with $i = 1, \ldots, m$.
\end{enumerate}

The quality of the feasible solution $x_1, \ldots, x_m$ for $\sdp_n$ is measured by the expectation
\begin{equation*}
\E\biggl[\sum_{i=1}^m\sum_{j=1}^m A_{ij} \; x_i \cdot x_j\biggr] = 
\sum_{i=1}^m\sum_{j=1}^m A_{ij} \E\biggl[\frac{Xu_i}{\|Xu_i\|} \cdot \frac{Xu_j}{\|Xu_j\|} \biggr],
\end{equation*}
which we analyze in more detail.

For vectors $u, v \in S^{\infty}$ we define
\begin{equation}
\label{eq:expectation}
E_n(u,v) = \E\biggl[\frac{Xu}{\|Xu\|} \cdot \frac{Xv}{\|Xv\|} \biggr],
\end{equation}
where $X = (X_{ij})$ is a matrix with $n$ rows and infinitely many columns whose entries are distributed independently according to the the standard normal distribution. Of course, if $u, v \in S^{m-1}$, then it suffices to work with finite matrices $X \in \R^{n \times m}$.

The first important property of the expectation $E_n$ is that it is
\defi{invariant under $\ort(\infty)$}, i.e.\ for every $m$ it is
invariant under the orthogonal group $\ort(m) = \{T \in \R^{m \times
  m}: T^{\sf T} T = I_m\}$, where $I_m$ denotes the identity
matrix. More specifically, for every $m$ and every pair of vectors $u,
v \in S^{m-1}$ we have
\begin{equation*}
E_n(Tu,Tv) = E_n(u,v)\qquad\text{for all $T \in \ort(m)$.}
\end{equation*}
If $n = 1$, then 
\begin{equation*}
E_1(u,v) = \E[\sign(\xi \cdot u) \sign(\xi \cdot v)],
\end{equation*}
where $\xi \in \R^m$ is chosen at random from the $m$-dimensional standard normal distribution. By
Grothendieck's identity (see e.g.\ \cite[Lemma 10.2]{Jameson})
\begin{equation*}
\E[\sign(\xi \cdot u) \sign(\xi \cdot v)] = \frac{2}{\pi} \arcsin u \cdot v.
\end{equation*}
Hence, the expectation $E_1$ only depends on the inner product $t = u \cdot v$.  
For general $n$, the $\ort(\infty)$ invariance implies that this is true also for $E_n$.
%Also for general $n$ the expectation $E_n$ is invariant under $\ort(\infty)$.

The second important property of the expectation $E_n$ (now
interpreted as a function of the inner product) is that it is a
function of positive type for $S^{\infty}$, i.e.\ it is of positive
type for any unit sphere $S^{m-1}$, independent of the dimension~$m$.
In general, a continuous function $f : [-1,1] \to \R$ is called
\defi{a function of positive type for $S^{m-1}$} if the matrix $(f(v_i
\cdot v_j))_{1 \leq i, j \leq N}$ is positive semidefinite for every
positive integer~$N$ and every choice of vectors $v_1, \ldots, v_N \in
S^{m-1}$.  The expectation $E_n$ is of positive type for $S^{\infty}$
because one can write it as a sum of squares.
Schoenberg~\cite{Schoenberg} characterized the continuous functions $f
: [-1,1] \to \R$ which are of positive type for $S^{\infty}$: They are
of the form
\begin{equation*}
f(t) = \sum_{i = 0}^\infty f_i t^i,
\end{equation*}
with nonnegative $f_i$ and $\sum_{i = 0}^{\infty} f_i < \infty$. In the case $n = 1$ we have the series expansion
\begin{equation*}
E_1(t) = \frac{2}{\pi} \arcsin t = \frac{2}{\pi} \sum_{i=0}^\infty \frac{(2i)!}{2^{2i}(i!)^2(2i+1)} t^{2i+1}.
\end{equation*}
In Section~\ref{sec:algorithm} we treat the cases $n \geq 2$.

Suppose we develop the expectation $E_n(t)$ into the series $E_n(t) = \sum_{i = 0}^{\infty} f_i t^i$. Then because of Schoenberg's characterization the function $t \mapsto E_n(t) - f_1 t$ is of positive type for $S^{\infty}$ as well. This together with the inequality $\sum_{i,j} X_{ij} Y_{ij} \geq 0$, which holds for all positive semidefinite matrices $X, Y \in \R^{m \times m}$, implies
\begin{equation}
\label{eq:approxineq}
\sdp_n(A) \geq \sum_{i=1}^m \sum_{j=1}^m A_{ij} E_n(u_i,u_j)  \geq  f_1 \sum_{i=1}^m \sum_{j=1}^m A_{ij} \; u_i \cdot u_j = f_1 \sdp_{\infty}(A).
\end{equation}
When $n = 1$ the series expansion of $E_1$ gives $f_1 = 2/\pi$ and the above argument is essentially the one of Nesterov~\cite{Nesterov}. To improve on this (and in this way to improve the constant $2/\pi$ in inequality~\eqref{eq:rietznesterov}) one can refine the analysis by working with functions of positive type which depend on the dimension $m$. In Section~\ref{sec:improved} we show that $t \mapsto 2/\pi (\arcsin t - t/\gamma(m))$ is a function of positive type for $S^{m-1}$.  For the cases $n \geq 2$ we show in Section~\ref{sec:algorithm} that $f_1 = \gamma(n)$.

\section{Analysis of the approximation algorithm}
\label{sec:algorithm}

In this section we show that the expectation $E_n$ defined in \eqref{eq:expectation} is a function of positive type for $S^{\infty}$ and that in the series expansion $E_n(t) = \sum_{i=0}^{\infty} f_i t^i$ one has $f_1 = \gamma(n)$. These two facts combined with the discussion in Section~\ref{sec:psd} imply Theorem~\ref{th:algorithm}.
Let $u, v \in S^{m-1}$ be unit vectors and let $X = (X_{ij}) \in \R^{n \times m}$ be a random matrix whose entries are independently sampled from the standard normal distribution. Because of the invariance under the orthogonal group, for computing $E_n(u,v)$ we may assume that $u$ and $v$ are of the form
\begin{equation*}
\begin{split}
& u = (\cos\theta, \sin\theta, 0,\ldots, 0)^{\sf T}\\
& v = (\cos\theta, -\sin\theta, 0, \ldots, 0)^{\sf T}.
\end{split}
\end{equation*}
Then by the double-angle formula $\cos 2\theta = t$ with $t = u \cdot v$. 

We have
\begin{equation*}
Xu =
\begin{pmatrix}
X_{11} & X_{12}\\
\vdots & \vdots\\
X_{n1} & X_{n2}
\end{pmatrix}
\begin{pmatrix}
\cos\theta\\
\sin\theta
\end{pmatrix},
\quad
Xv =
\begin{pmatrix}
X_{11} & X_{12}\\
\vdots & \vdots\\
X_{n1} & X_{n2}
\end{pmatrix}
\begin{pmatrix}
\cos\theta\\
-\sin\theta
\end{pmatrix}.
\end{equation*}
Hence,
\begin{equation*}
\frac{Xu}{\|Xu\|} \cdot \frac{Xv}{\|Xv\|}  =
\frac{x^{\sf T} Y y}{\sqrt{(x^{\sf T}Yx) (y^{\sf T}Y y)}},
\end{equation*}
where $x = (\cos\theta,\sin\theta)^{\sf T}$, $y = (\cos\theta,-\sin\theta)^{\sf T}$, and $Y \in \R^{2 \times 2}$ is the Gram matrix of the two vectors $(X_{11}, \ldots, X_{n1})^{\sf T}$, $(X_{12}, \ldots, X_{n2})^{\sf T} \in \R^n$. 
By definition, $Y$ is distributed according to the Wishart distribution from multivariate statistics. This distribution is defined as follows (see e.g.\ Muirhead \cite{Muirhead}). Let $p$ and $q$ be positive integers so that $p \geq q$. The (standard) \defi{Wishart distribution}  $W_q(p)$ is the probability distribution of random matrices $Y = X^{\sf T} X \in \R^{q \times q}$, where the entries of the matrix $X = (X_{ij}) \in\R^{p\times q}$ are independently chosen from the standard normal distribution $X_{ij} \sim N(0,1)$. The density function of $Y \sim W_q(p)$~is
\begin{equation*}
\frac{1}{2^{pq/2} \Gamma_q(p/2)}e^{-\Tr(Y)/2}(\det Y)^{(p-q-1)/2},
\end{equation*}
where $\Gamma_q$ is the \defi{multivariate gamma function}, defined as
\begin{equation*}
\Gamma_q(x) = \pi^{q(q-1)/4}\prod_{i=1}^q\Gamma\Big(x - \frac{i-1}{2}\Big).
\end{equation*}

We denote the cone of positive semidefinite matrices of size $q \times q$ by $S^q_{\geq 0}$. In our case $p = n$ and $q = 2$. We can write $E_n(t)$ as
\begin{equation*}
E_n(t) = \frac{1}{2^n \Gamma_2(n/2)} \int_{S^{2}_{\geq 0}} 
\frac{x^{\sf T} Y y}{\sqrt{(x^{\sf T}Yx) (y^{\sf T} Y y)}}
e^{-\Tr(Y)/2} (\det Y)^{(n-3)/2} dY,
\end{equation*}
where $t =\cos 2\theta$, and $x$ as well as $y$ depend on $\theta$. The parameterization of  the cone $S^2_{\geq 0}$ given by
\begin{equation*}
\label{psdmatrix}
S^2_{\geq 0} =
\left\{
Y = \begin{pmatrix} \frac{a}{2} + \alpha\cos\phi & \alpha\sin\phi\\ \alpha\sin\phi & \frac{a}{2} - \alpha\cos\phi  \end{pmatrix} : \phi \in [0,2\pi], \alpha \in [0,a/2], a \in \R_{\geq 0}\right\}
\end{equation*}
allows us to write the integral in a more explicit form. With this parametrization we have 
\begin{equation*}
\Tr(Y) = a, \quad \det(Y) = \frac{a^2}{4} - \alpha^2, \quad dY =\alpha \; d\phi d\alpha da,
\end{equation*}
and
\begin{equation*}
\begin{split}
& x^{\sf T}Yy = \frac{at}{2} + \alpha\cos\phi,\\
& x^{\sf T}Yx = \frac{a}{2} + \alpha(t\cos\phi + 2\sin\theta\cos\theta\sin\phi),\\
& y^{\sf T}Yy = \frac{a}{2} + \alpha(t\cos\phi - 2\sin\theta\cos\theta\sin\phi).\\
\end{split}
\end{equation*}
So,
\begin{equation*}
\begin{split}
E_n(t) = \frac{1}{2^n\Gamma_2(n/2)} \int_{0}^{\infty} \int_0^{a/2} \int_{0}^{2\pi} \frac{\frac{at}{2}+\alpha\cos\phi}{\sqrt{(\frac{a}{2}+\alpha t\cos\phi)^2-\alpha^2(1-t^2)(\sin \phi)^2}}\qquad&\\
{}\cdot e^{-a/2} \left(\frac{a^2}{4} - \alpha^2\right)^{(n-3)/2}  \alpha \; d\phi d\alpha da.&
\end{split}
\end{equation*}
Substituting $\alpha = (a/2)r$ and integrating over $a$ yields
\begin{equation*} 
E_n(t) = \frac{\Gamma(n)}{2^{n-1}\Gamma_2(n/2)} \int_{0}^1 \int_{0}^{2\pi}\frac{(t+r\cos\phi)r (1-r^2)^{(n-3)/2}}{\sqrt{(1+rt\cos\phi)^2-r^2(1-t^2)(\sin\phi)^2}} d\phi dr.
\end{equation*}
Using Legendre's duplication formula (see \cite[Theorem 1.5.1]{AndrewsAskeyRoy}) 
$
\Gamma(2x)\Gamma(1/2) = 2^{2x-1}\Gamma(x)\Gamma(x+1/2)
$
one can simplify
\begin{equation*}
\frac{\Gamma(n)}{2^{n-1}\Gamma_2(n/2)}  = \frac{n-1}{2\pi}.
\end{equation*}
Recall from \eqref{eq:approxineq} that the approximation ratio is given by the coefficient $f_1$ in the series expansion $E_n(t) = \sum_{i=0}^{\infty} f_i t^i$. Now we compute $f_1$:
\begin{eqnarray*}
f_1 & = & \frac{\partial E_n}{\partial t}(0)\\
& = & \frac{n-1}{2\pi} \int_{0}^1 \int_{0}^{2\pi} \frac{r(1-r^2)^{(n-1)/2}}{(1 - r^2(\sin \phi)^2)^{3/2}}d\phi dr.\\
\end{eqnarray*}
Using Euler's integral representation of the hypergeometric function \cite[Theorem 2.2.1]{AndrewsAskeyRoy} and by substitution we get
\begin{eqnarray*}
f_1 & = & \frac{n-1}{2\pi}\int_{0}^{2\pi} \frac{\Gamma(1)\Gamma((n+1)/2)}{2\Gamma((n+3)/2)} 
{}_{2}F_{1} \left(\begin{array}{cc} 3/2, 1\\ (n+3)/2\end{array}; \sin^2\phi \right) d\phi\\
& = & \frac{n-1}{4\pi} \frac{\Gamma((n+1)/2)}{\Gamma((n+3)/2)} 4 \int_0^1 
{}_{2}F_{1} \left(\begin{array}{cc} 3/2, 1\\ (n+3)/2\end{array}; t^2 \right) (1-t^2)^{-1/2} dt\\
& = & \frac{n-1}{\pi} \frac{\Gamma((n+1)/2)}{\Gamma((n+3)/2)}\frac{1}{2} \int_0^1 
{}_{2}F_{1} \left(\begin{array}{cc} 3/2, 1\\ (n+3)/2\end{array}; t \right)
(1-t)^{-1/2} t^{-1/2  }dt.
\end{eqnarray*}
This simplies futher by Euler's generalized integral \cite[(2.2.2)]{AndrewsAskeyRoy}, and Gauss's summation formula \cite[Theorem 2.2.2]{AndrewsAskeyRoy}
\begin{eqnarray*}
f_1 & = & \frac{n-1}{2\pi} \frac{\Gamma((n+1)/2)}{\Gamma((n+3)/2)} \frac{\Gamma(1/2)\Gamma(1/2)}{\Gamma(1)}
{}_{3}F_{2} \left(\begin{array}{cc} 3/2, 1,1/2\\ (n+3)/2,1\end{array}; 1 \right)\\
& = & \frac{n-1}{2} \frac{\Gamma((n+1)/2)}{\Gamma((n+3)/2)}
{}_{2}F_{1} \left(\begin{array}{cc} 3/2, 1/2\\ (n+3)/2\end{array}; 1 \right)\\
& = & \frac{n-1}{2} \frac{\Gamma((n+1)/2)}{\Gamma((n+3)/2)} \frac{\Gamma((n+3)/2)\Gamma((n-1)/2)}{\Gamma(n/2)\Gamma((n+2)/2)}\\
& = & \frac{2}{n}\left(\frac{\Gamma((n+1)/2)}{\Gamma(n/2)}\right)^2.
\end{eqnarray*}

%\begin{remark}
%Haagerup \cite{Haagerup} computed the function $E_2$ explicitly in terms of elliptic integrals of the first and second kind, $K$ and $E$:
%\begin{equation*}
%\begin{split}
%E_2(t) & = t \int_0^{\pi/2} \frac{(\cos \theta)^2}{\sqrt{1-t^2(\sin \theta)^2}} d\theta = \frac{1}{t}\left(E(t)-(1-t^2)K(t)\right)\\
%& = \frac{\pi}{4}
%\left( 
%t + \left(\frac{1}{2}\right)^2 \frac{t^3}{2} 
%+ \left(\frac{1 \cdot 3}{2 \cdot 4}\right)^2 \frac{t^5}{3}  + \left(\frac{1\cdot 3 \cdot 5}{2 \cdot 4 \cdot 6}\right)^2 \frac{t^7}{4} + \cdots\right).
%\end{split}
%\end{equation*}
% (Note that on page 201 in \cite{Haagerup} $\pi/2$ has to be $\pi/4$.)
%\end{remark}

%
%
%

\section{Improved analysis}
\label{sec:improved}

Nesterov's proof of inequality \eqref{eq:rietznesterov} relies on the
fact that the function $t \mapsto 2/\pi (\arcsin t - t)$ is of
positive type for $S^{\infty}$. Now we determine the largest
value~$c(m)$ so that the function $t \mapsto 2/\pi (\arcsin t - c(m)
t)$ is of positive type for $S^{m-1}$. By
this we improve the approximation ratio of the algorithm given in
Section~\ref{sec:psd} for $\sdp_1$ from $2/\pi$ to $(2/\pi) c(m)$. The
following lemma showing $c(m) = 1/\gamma(m)$ implies
Theorem~\ref{th:improvement}.

\begin{lemma}
\label{lem:psd}
The function
\begin{equation*}
t \mapsto \frac{2}{\pi} \left(\arcsin t - \frac{t}{\gamma(m)} \right)
\end{equation*}
is of positive type for $S^{m-1}$.
\end{lemma}

\begin{proof}
We equip the space of all continuous functions $f : [-1,1] \to \R$ with the inner product
\begin{equation*}
(f, g)_{\alpha} = \int_{-1}^1 f(t) g(t) (1-t^2)^\alpha dt, 
\end{equation*}
where $\alpha = (m-3)/2$. With this inner product the Jacobi polynomials satisfy the orthogonality relation
\begin{equation*}
(P_i^{(\alpha,\alpha)}, P_j^{(\alpha,\alpha)})_{\alpha} = 0, \quad \text{if $i \neq j$},
\end{equation*}
where $P_i^{(\alpha,\alpha)}$ is the Jacobi polynomial of degree $i$ with parameters $(\alpha, \alpha)$, see e.g.\ Andrews, Askey, and Roy \cite{AndrewsAskeyRoy}. 
Schoenberg \cite{Schoenberg} showed that a continuous function $f : [-1,1] \to \R$ is of positive type for $S^{m-1}$ if and only if it is of the form 
\begin{equation*}
f(t) = \sum_{i = 0}^\infty f_i P_i^{(\alpha,\alpha)}(t),
\end{equation*}
with nonnegative coefficients $f_i$ such that $\sum_{i = 0}^{\infty} f_i < \infty$. 

Now we interpret $\arcsin$ as a function of positive type for $S^{m-1}$ where $m$ is fixed. By the orthogonality relation and because of Schoenberg's result the function $\arcsin t - c(m)t$ is of positive type for $S^{m-1}$ if and only if
\begin{equation*}
(\arcsin t - c(m) t, P_i^{(\alpha,\alpha)})_{\alpha} \geq 0,\qquad\text{for all $i = 0$, $1$, $2$, \dots.}
\end{equation*}
We have $P_1^{(\alpha,\alpha)}(t) = (\alpha + 1)t$. By the orthogonality relation and because the $\arcsin$ function is of positive type we get, for $i \ne 1$,
\begin{equation*}
(\arcsin t -c(m)t , P_i^{(\alpha,\alpha)})_{\alpha} = (\arcsin t, P_i^{(\alpha,\alpha)})_{\alpha} \ge 0.
\end{equation*}
This implies that the maximum $c(m)$ such that $\arcsin t - c(m)t$ is of positive type for $S^{m-1}$ is given by $c(m) = (\arcsin t, t)_{\alpha}/(t,t)_{\alpha}$.

The numerator of $c(m)$ equals
\begin{equation*}
\begin{split}
(\arcsin t, t)_{\alpha} & = \int_{-1}^1 \arcsin(t) t (1-t^2)^{\alpha} dt  \\
& = \int_{-\pi/2}^{\pi/2} \theta \sin \theta (\cos \theta)^{2\alpha+1} d\theta\\
& = \frac{\Gamma(1/2)\Gamma(a+3/2)}{(2\alpha+2) \Gamma(\alpha + 2)}.
\end{split}
\end{equation*}
The denominator of $c(m)$ equals
\begin{equation*}
(t,t)_{\alpha} = \int_{-1}^1 t^2 (1-t^2)^{\alpha} dt = \frac{\Gamma(3/2) \Gamma(\alpha+1)}{\Gamma(\alpha + 5/2)},
\end{equation*}
where we used the beta integral (see e.g.\ Andrews, Askey, and Roy \cite[(1.1.21)]{AndrewsAskeyRoy})
\begin{equation*}
\int_{0}^1t^{2x-1}(1-t^2)^{y-1}dr = \int_0^{\pi/2} (\sin \theta)^{2x-1}(\cos \theta)^{2y-1}d\theta = \frac{\Gamma(x)\Gamma(y)}{2\Gamma(x+y)},
\end{equation*}
Now, by using the functional equation $x \Gamma(x) = \Gamma(x+1)$, the desired equality $c(m) = 1/\gamma(m)$ follows.
\qed
\end{proof}

\section{Hardness of approximation}
\label{sec:hardness}

\begin{proof}[of Theorem~\ref{th:hardness}]
Suppose that $\rho$ is the largest approximation ratio a poly\-nomial-time algorithm can achieve for $\sdp_n$. Let $u_1, \ldots, u_m \in S^{n-1}$ be an approximate solution to $\sdp_n(A)$ coming from such a poly\-nomial-time algorithm. Then,
\begin{equation*}
\sum_{i=1}^m \sum_{j=1}^m A_{ij} \; u_i \cdot u_j \geq \rho \sdp_n(A).
\end{equation*}
Applying the rounding scheme to $u_1, \ldots, u_m \in S^{n-1}$ gives $x_1, \ldots, x_m \in \{-1,+1\}$  with
\begin{equation*}
\begin{split}
\E\biggl[\sum_{i=1}^m\sum_{j=1}^m A_{ij} \; x_i x_j\biggr] & = \frac{2}{\pi}
\sum_{i=1}^m \sum_{j=1}^m A_{ij} \arcsin u_i \cdot u_j\\
&  \geq \frac{2 \rho}{\pi \gamma(n)}\sdp_n(A),
\end{split}
\end{equation*}
where we used that the matrix $A$  and the matrix
\begin{equation*}
{\left(\frac{2}{\pi}\left(\arcsin u_i \cdot u_j - \frac{u_i \cdot u_j}{\gamma(n)}\right)\right)}_{1 \leq i,j \leq m}
\end{equation*}
are both positive semidefinite. The last statement follows from Lemma~\ref{lem:psd} applied to the vectors $u_1, \ldots, u_m $ lying in $S^{n-1}$. Since $\sdp_n(A) \geq \sdp_1(A)$, this is a polynomial-time approximation algorithm for $\sdp_1$ with approximation ratio at least $(2\rho)/(\pi\gamma(n))$. The UGC hardness result of Khot and Naor now implies that $\rho \leq \gamma(n)$.
\qed
\end{proof}

\section{The case of Laplacian matrices}
\label{sec:Laplacian}

In this section we show that one can improve the approximation ratio of the algorithm if the positive semidefinite matrix $A = (A_{ij}) \in \R^{m \times m}$ has the following special structure:
\begin{equation*}
\begin{split}
A_{ij} \leq 0, & \quad \text{if $i \neq j$,}\\
\sum_{i = 1}^n A_{ij} = 0, & \quad \text{for every $j = 1, \ldots, n$}.
\end{split}
\end{equation*}
This happens for instance when $A$ is the Laplacian matrix of a weighted graph with nonnegative edge weights. A by now standard argument due to Goemans and Williamson \cite{GoemansWilliamson} shows that the algorithm has the approximation ratio
\begin{equation*}
v(n) = \min\left\{\frac{1-E_n(t)}{1-t} : t \in [-1,1]\right\}.
\end{equation*}

To see this, we write out the expected value of the approximation and use the properties of $A$:
\begin{eqnarray*}
\E\Big[\sum_{i,j=1}^nA_{ij}x_i\cdot x_j\Big] &=& \sum_{i,j=1}^nA_{ij}E_n(u_i\cdot u_j)\\
&=& \sum_{i\not= j}(-A_{ij})\left(\frac{1 - E_k(u_i\cdot u_j)}{1-u_i\cdot u_j}\right)(1-u_i\cdot u_j)\\
&\geq& v(n)\sum_{i, j=1}^nA_{ij}u_i\cdot u_j\\
&=&  v(n)\sdp_{\infty}(A).
\end{eqnarray*}

The case $n = 1$ corresponds to the MAX CUT approximation algorithm of Goemans and Williamson \cite{GoemansWilliamson}.  For this we have
\begin{equation*}
v(1) = 0.8785\dots,\quad \text{minimum attained at $t_0 = -0.689\dots$.}
\end{equation*}
We computed the values $v(2)$ and $v(3)$ numerically and got
\begin{equation*}
\begin{split}
& v(2) = 0.9349\dots,\quad \text{minimum attained at $t_0 = -0.617\dots$,}\\
& v(3) = 0.9563\dots, \quad \text{minimum attained at $t_0 = -0.584\dots$.}\\
\end{split}
\end{equation*}

\section*{Acknowledgements}
 
We thank Joe Eaton, Monique Laurent, Robb Muirhead, and Achill Sch\"urmann for helpful discussions. The third auhor thanks the Institute for Pure \& Applied Mathematics at UCLA for its hospitality and support.

\end{document}